\documentclass[conference]{IEEEtran}
\IEEEoverridecommandlockouts
\usepackage{cite, graphicx, xcolor, enumerate, hyperref, amssymb}
\def\BibTeX{{\rm B\kern-.05em{\sc i\kern-.025em b}\kern-.08em
    T\kern-.1667em\lower.7ex\hbox{E}\kern-.125emX}}
\definecolor{NCSUred}{RGB}{153, 0, 0}
\definecolor{NCSUgreen}{RGB}{0, 132, 115}
\definecolor{NCSUblue}{RGB}{65, 86, 161}
\definecolor{NCSUorange}{RGB}{209, 73, 5}

\hypersetup{colorlinks=true, breaklinks=true, plainpages=false, citecolor=NCSUblue, linkcolor=NCSUblue, urlcolor=NCSUorange}
\usepackage{amsmath, amssymb, amsthm, amsfonts}
\usepackage[notext]{stix}
\usepackage[cal=rsfso, bb=ams]{mathalpha}
\usepackage{algorithmic} 
\newtheorem{theorem}{Theorem}
\newtheorem{definition}{Definition}

\newtheorem{lemma}{Lemma}

\newtheorem{problem}{Problem}
\newcommand{\mc}[1]{\mathcal{#1}}
\newcommand{\ms}[1]{\mathsf{#1}}
\newcommand{\mr}[1]{\mathrm{#1}}

\newcommand{\mbb}[1]{\mathbb{#1}}

\newcommand{\rg}[1]{\mathring{#1}}
\newcommand{\mR}{\mathbb{R}}

\newcommand{\mN}{\mathbb{N}}
\newcommand{\xD}[1]{\mr{d} #1}

\newcommand{\bra}[1]{\left( #1 \right)}

\newcommand{\norm}[1]{\left\| #1 \right\|}
\newcommand{\ip}[2]{\left\langle #1, \, #2 \right\rangle}
\newcommand{\vk}{\varkappa}

\newcommand{\vf}{\varphi}

\begin{document}
\title{Koopman Modeling and Stabilization of Discrete-Time Nonlinear Control Systems: Bilinearity on a Reproducing Kernel Hilbert Space \\  
\thanks{The work is supported by NSF under CBET Award \#2414369. All simulation codes for the numerical experiments in this paper are available at this hyperlinked \href{https://github.com/XiuzhenYe/Koopman-Modeling-and-Controller-Synthesis-for-Bilinear-Systems-on-Reproducing-Kernel-Hilbert-Space}{GitHub repository}.}}
\author{\IEEEauthorblockN{Jarod Morris, Xiuzhen Ye, Wentao Tang}
\IEEEauthorblockA{\textit{Department of Chemical and Biomolecular Engineering, North Carolina State University}\\
Raleigh, North Carolina, U.S.A. \\ 
\texttt{\{jemorri4, xye7, wtang23\}@ncsu.edu}}
% \and
}

\maketitle
\begin{abstract}
Despite the popularity of Koopman modeling for nonlinear systems, in the presence of input variables, the evident nonexistence of a fully linear time-invariant model even in infinite dimensions makes Koopman-based control largely an open problem to date. 
Focusing on discrete-time systems in this paper, which eschews from using operator semigroup and infinitesimal generator notions, it is proven that nonlinear systems, if satisfying appropriate smoothness and regularity conditions, can be expressed \emph{exactly as bilinear dynamics}, when the state variables and input variables are separately lifted into their reproducing kernel Hilbert spaces (RKHSs). 
To account for the knowledge of an equilibrium point at the origin, the RKHS is defined by a linear--radial product kernel, and hence the functions belonging to this RKHS are spanned by the multiplications of component functions and Sobolev functions. 
The stabilization problem, namely the determination of a feedback law that causes a Lyapunov function (expressed as a kernel sum-of-squares form) to decrease, is then posed as an infinite-dimensional optimization problem over state-dependent conditional probability measures over the input space, solved via a discretization scheme. 
\end{abstract}


\section{Introduction}
The interest in Koopman operator theory once gave a hope that nonlinear control problems, if using some state lifting, can be converted to that for a linear time-invariant (LTI) system \cite{proctor2018generalizing, korda2018linear}. 
The shortcoming in this ambiguous LTI postulate and its misleading effect on later literature is evident --- given a nonlinear system with inputs in discrete time: $x_{t+1} = f(x_t, u_t)$, hereafter simply written as 
\begin{equation}\label{eq:system}
    x^+ = f(x, u), \enspace f: \mR^n \times \mR^m \to \mR^n,
\end{equation}
there is clearly no finite-dimensional ``lifting'' $z = \phi(x)$ such that $z^+=Az+Bu$ can be satisfied for some matrices $A$ and $B$. 
Indeed, since the Koopman operator $K: g\mapsto g\circ f$ for the autonomous system $x^+ = f(x)$ generally must be defined on an infinite-dimensional function space, except for very restrictive cases \cite{shang2026existence}, the Koopman model for \eqref{eq:system} cannot be a finite-dimensional LTI one. 
In fact, even in infinite dimensions, an LTI representation is absent, as the effect of inputs on the lifted states cannot be expected to be a constant gain. 
Such an observation has been well pointed out by practitioners \cite{bruder2021advantages}, which motivated the use of bilinear Koopman models \cite{goswami2021bilinearization} among other types of customized linear Koopman models, e.g., for Volterra systems or Hammerstein systems \cite{hou2026behavioral}. 

\par The bilinearity can be interpreted in the following way. Considering the Koopman operator defined for \eqref{eq:system}, $K: g\mapsto g\circ f$, it maps from a space of state-dependent functions into a state-and-input-dependent functions; for the latter space, any function should be spanned by features of states and inputs altogether. 
Formally, this can be expressed in the language of reproducing kernel Hilbert spaces (RKHS) --- in Bevanda et al. \cite{bevanda2024nonparametric}, the \emph{control Koopman operator} is defined as a mapping between a state-RKHS and a tensor product of state-RKHS and input-RKHS. 
Another approach to formulate a bilinear Koopman model is to use continuous-time input-affine systems $\dot{x} = f(x) + G(x)u$, where, upon taking a continuously differentiable lifting map $z=\phi(x)$, gives $\dot{z} = (\mr{D}\phi\cdot f)(x) + (\mr{D}\phi\cdot G)(x)u$. This route, however, often cannot evade from theoretical inconsistencies,\footnote{In continuous time, $\phi\mapsto \mr{D}\phi\cdot f$ and $\phi\mapsto \mr{D}\phi\cdot G$ must be understood as infinitesimal generators of operator semigroups, and these operators are themselves unbounded and only densely defined. The meaning of a bilinear system $\dot{z} = Az+ \sum_j u_j(B_jz)$ on infinite-dimensions turns out to be ambiguous, without a proper definition for mild solutions when $A$ and $B$ are unbounded operators, possibly deefined on different spaces \cite{iacob2024koopman}. Some recent literature that aimed at using bilinear Koopman models for control eluded from such discussions, and retreated to the assumption of a finite-dimensional exact embedding \cite{strasser2024koopman}. Given that discrete-time Koopman operator models are more computationally friendly, one tends to actually deal with discrete-time systems. However, when discretizing the time, input affinity cannot be preserved. \label{fn:1}} 
and therefore we leave the case of continuous-time systems to future discussions.

\par To the end of a universal Koopman representation of discrete-time system \eqref{eq:system}, we believe that the following aspects should be rigorously addressed without ambiguity. 

\paragraph{Choice of spaces of state-dependent and state-and-input-dependent functions}
The domain and codomain of the Koopman operator $\mc{G}_x \ni g$ and $\mc{G}_{(x,u)}\ni g\circ f$ should be so that $K$ is indeed a bounded linear operator. For the learning purpose, these spaces are desired to be RKHSs. 
The condition for the operator $K$ of an autonomous systems to be a bounded linear operator on a \emph{Sobolev-type RKHS} (generated by a Sobolev radial kernel) was given in \cite{kohne2025error}. 
Noting that equilibrium point plays a special role in dynamical systems and needs to be preserved in the Koopman model for stability analysis, our prior work \cite{tang-ye2025koopman} used a linear--radial kernel (a linear kernel multiplied by a Sobolev kernel) to define a RKHS, and gave corresponding conditions for the Koopman operator to be bounded. 
These conditions are easy to extend to non-autonomous system \eqref{eq:system}, if simplistically viewing $(x,u)$ as a joint variable \cite{tang-ye2026dissipativity}. 

\paragraph{Separation of state features and input features}
Viewing $(x, u)$ as a joint variable in the dynamics $f$, as considered in \cite{tang-ye2026dissipativity}, however, restricts the capacity of handling $u$ as a decision variable. 
The limitation remains if the inputs are regarded only as parameters in a parameterized Koopman operator family \cite{haseli2025two}. 
The tensor product formulation, which allows the features of states and the features of inputs to be separated into their own spaces, is desirable and has been adopted in some relevant papers (e.g., \cite{bevanda2024kernel, lazar2025product}, as well as \cite{tang2025koopman} considering the closed-loop system using ``policy kernels''). 
The conditions for meaningfully defining an infinite-dimensional Koopman operator model involving the tensor product RKHS, however, remains to be clarified. In particular, we are interested in finding the conditions for an exact bilinear Koopman model on RKHSs, rather than a ``misspecified'' approximate model (e.g., from RKHS to $L^2$) as in \cite{klus2020eigendecompositions, kostic2022learning, bevanda2023koopman, hou2024nonparametric}. 
Clearly, to handle problems such as multi-step prediction and control, it is necessary to establish the operator on the same function space. 

\paragraph{Explicit solution of typical control problems}
Due to the nonlinearity of the input feature map, the control problems of \eqref{eq:system} remain intrinsically nonlinear and cannot na{\"{i}}vely follow a linear-quadratic control formulation. 
As a result, by using receding horizon optimization as an implicit control law, existing literature \cite{bevanda2024kernel, bold2025kernel} focused on model predictive control (MPC) formulations for bilinear Koopman models. 
In such settings, recursive feasibility and closed-loop stability guarantees can be deduced only via \emph{ad hoc} assumptions on cost functions and terminal constraint. Essentially, the MPC approach forgoes the \emph{synthesis of an explicit controller}, which ought to be, at least partially, the rationale for seeking a Koopman model of control systems. 
To the authors' knowledge, a rigorous controller \emph{synthesis} approach based on a general Koopman model does not yet exist. 

\par The aim of this paper is to establish that the system \eqref{eq:system}, under mild regularity assumptions on the dynamics $f$, can be represented by an \emph{exact bilinear Koopman model} (see equation \eqref{eq:bilinear}), where the state $x$ and input $u$ are separately lifted into their canonical features in RKHSs associated with linear--radial kernels \cite{tang-ye2025koopman}), and the infinite-dimensional model involves the Kronecker product of their features that reside in the tensor product RKHS. 
The identification of such a bilinear model naturally lends to a kernel regression routine in the spirit of \cite{kostic2022learning}. 
Furthermore, under the bilinear model \eqref{eq:bilinear}, a stabilizing feedback law is sought by searching a probability distribution over the input space, dependent on the state, that renders a Lyapunov function, expressed as a kernel sum-of-squares (SOS) form, to decay. The problem turns out to be convex and therefore computationally tractable by approximation, whose solution provides a \emph{stabilizing controller synthesis} method.

\section{Preliminaries}\label{sec:preliminaries}

\subsection{RKHS and Sobolev Kernel}
\par We call a continuous function $\kappa:\mbb{X}\times\mbb{X}\rightarrow\mR$ a \emph{kernel} on $\mbb{X}$ if for any finite set of points $\left\{ x_i \right\}_{i=1}^N \subset \mbb{X}$, the matrix defined as $G_\kappa=\left[ \kappa\left(x_i, x_j\right) \right]_{i,j=1}^N$ is symmetric and positive semidefinite. The RKHS associated with $\kappa$ is defined as
\begin{equation*}
    \mathcal{H}_\kappa(\mbb{X}) = \overline{\operatorname{span}} \left\{ \kappa(x,\cdot) : x \in \mathbb{X} \right\},
\end{equation*}
and its elements $\phi_x:=\kappa(x,\cdot)$ are called kernel functions or canonical features. The inner product of $\mc{H}_\kappa(\mbb{X})$ is defined as $\ip{\kappa(x,\cdot)}{\kappa(x',\cdot)}=\kappa(x,x')$, and consequently the induced norm is $\lVert \kappa(x,\cdot)\rVert^2=\kappa(x,x)$. The RKHS has the \emph{reproducing kernel} property: for any $g\in\mc{H}_\kappa(\mbb{X})$, the point evaluation at any point $x\in\mbb{X}$ can be written as the inner product with a kernel function, i.e., $g(x)=\ip{g}{\kappa(x,\cdot)}$. Similarly, a kernel function function $\vk:\mbb{U}\times\mbb{U}\rightarrow\mR$ and its corresponding RKHS can be defined over the input space $\mbb{U}$, denoted as $\mc{H}_\vk(\mbb{U})$.

\par A useful type of Hilbert space is a Sobolev--Hilbert space $H^s(\mbb{X}):=W^{s,2}(\mbb{X})$, which is the space of such functions that have generalized derivatives up to $s$-th order, all of which being square-integrable over $\mbb{X}$. 
This regularity is important for Koopman analysis, because Koopman operator is known to be well-defined on the Sobolev space under smoothness and nondegeneracy conditions on the dynamics \cite{kohne2025error}. 
\begin{lemma}[Wendland \cite{wendland2004scattered}]\label{lem:wendland}
    Let $\kappa$ be a radial kernel on $\mbb{X}$, that is, a kernel such that $\kappa(x,x')=\rho\left(\lvert x-x'\rvert\right)$ for some $\rho:\mR_+\rightarrow\mR_+$ with a Fourier transform, $\hat \rho (\xi)$, that satisfies $c_1\left(1+\lvert \xi \rvert^2 \right)^{-s/2} \leq |\hat \rho (\xi)| \leq c_2 \left( 1+\lvert \xi \rvert^2 \right)^{-s/2}$ for some constants $c_2 \geq c_1 > 0$ for all $\xi \in \mR$. If $\mbb{X}$ further has a Lipschitz boundary, then $W^{s,2}(\mbb{X})=\mc{H}_\kappa(\mbb{X)}$ up to an equivalence of norms.
\end{lemma}

\par Suppose $\kappa:\mbb{X}\times\mbb{X}\rightarrow\mR$ and $\vk:\mbb{U} \times \mbb{U} \rightarrow \mR$ are kernels satisfying the conditions in Lemma \ref{lem:wendland}. Then we define $\mc{H}_x:=\mc{H}_\kappa(\mbb{X})$ and $\mc{H}_u:=\mc{H}_\vk(\mbb{U})$ as the Sobolev--Hilbert spaces over the state and input spaces.
We also introduce the notion of an adjoint operator. Consider a well-defined, bounded linear operator $T: \mc{G}\rightarrow\mc{H}$ between Hilbert spaces. The adjoint operator, denoted $T^*:\mc{H}\to\mc{G}$, is defined by the relation $\ip{T^*h}{g}=\ip{h}{Tg}$ for all $g\in \mc{G}$ and $h\in\mc{H}$. 
It is easily provable that the adjoint of the Koopman operator, denoted $K^*$ and also called the Perron--Frobenius operator, in the case of an autonomous system $x^+=f(x)$, gives
\begin{equation*}
    K^*\rg{\kappa}(x,\cdot) = \rg{\kappa}(f(x),\cdot), \qquad \forall x\in\mbb{X}.
\end{equation*}

\subsection{Tensor Product and Linear--Radial Kernel}
\par Sobolev-type RKHSs are useful for Koopman analysis as they provide a regular function space on which composition with dynamics can be a well-defined, bounded linear operator. However, a radial Sobolev kernel is insensitive to the special role of an \emph{equilibrium point} (at the origin, without loss of generality). 
Specifically, if the kernel has the form of $\kappa(x,x')=\rho\left(\vert x-x'\rvert\right)$, then $\left\lVert \kappa(x,\cdot) \right\rVert_{\mc{H}_x}^2 = \kappa(x,x) = \rho(0)$ for every $x\in \mbb{X}$. Hence the canonical feature at each point has the same norm, and the origin in $\mbb{X}$ does not correspond to the origin in the RKHS. The same undesirable obstruction appears in the controlled setting.

\par We recall the tensor product spaces of Hilbert spaces, which will be used to both define the linear--radial kernel and the state--input feature used in the bilinear Koopman representation. 
Suppose that $\mc{H}_1$ and $\mc{H}_2$ are two Hilbert spaces. Their tensor product is defined as
\begin{equation*}
    \mc{H}_1 \otimes \mc{H}_2 = \overline{\operatorname{span}}\left\{ g \otimes h: g\in\mc{H}_1, ~ h\in\mc{H}_2\right\},
\end{equation*}
with the elementary tensors of the form $g\otimes h$ are ordered pairs in $\mc{H}_1\times \mc{H}_2$ endowed with an inner product $\ip{g_1 \otimes h_1}{g_2 \otimes h_2} = \ip{g_1}{g_2}_{\mc{H}_1} \ip{h_1}{h_2}_{\mc{H}_2}$, and then extended by linearity and completion. 
When $\mc{H}_1$ and $\mc{H}_2$ are further RKHSs, $\mc{H}_{\kappa_1}(\Omega_1)$ and $\mc{H}_{\kappa_2}(\Omega_2)$, respectively, their tensor product is again an RKHS, whose reproducing kernel is simply the product kernel $\kappa((\omega_1, \omega_2), (\cdot,\cdot))=\kappa_1(\omega_1, \cdot) \kappa_2(\omega_2, \cdot)$. That is, $\mc{H}_{\kappa_1}(\Omega_1)\otimes \mc{H}_{\kappa_2}(\Omega_2) = \mc{H}_\kappa(\Omega_1\times\Omega_2)$. 
If $\Omega_1=\Omega_2=\Omega$, we can just force $\omega_1=\omega_2=\omega\in\Omega$, namely define $\kappa = \kappa_1\kappa_2$ and then $\mc{H}_{\kappa_1}(\Omega)\otimes \mc{H}_{\kappa_2}(\Omega) = \mc{H}_\kappa(\Omega)$.  

\par Using this construction, the linear--radial kernel was introduced in our previous work \cite{tang-ye2025koopman}. Denote $\kappa_{\mr{sob}}(x,x')=\rho\left(\vert x - x' \rvert\right)$ as the radial Sobolev kernel on $\mbb{X}$. We define the linear kernel $\kappa_\mr{lin}(x,x') = x^\top x'$ and form the product kernel 
\begin{equation*}
    \rg\kappa(x,x')=\kappa_\mr{lin}(x,x')\kappa_{\mr{sob}}(x,x') = \left( x^\top x' \right) \rho\left(\lvert x - x'\rvert\right).
\end{equation*}
\begin{lemma}[Tang \& Ye \cite{tang-ye2025koopman}]
    The function $\rg\kappa$ defines an RKHS over $\mbb{X}$, denoted $\rg{\mc{H}}_x$, can be expressed as:  
    $$ \rg{\mc{H}}_x=\left\{ \sum_{k=1}^n e_kg_k : g_k \in \mc{H}_x \right\}, \norm{ \sum_{k=1}^n e_kg_k }_{\rg{\mc{H}}_x}^2 = \sum_{k=1}^n \norm{g_k}_{\mc{H}_x}^2, $$
    where $e_k(x)=x_k$ ($k=1,\cdots,n)$ are the projection maps.  
\end{lemma}

\par Informally, this space comprises of functions that are \emph{locally  at least linear} near the equilibrium point, while still globally regular in the Sobolev sense. Since the linear factors now make the canonical feature vanish at the origin, we have $\rg{\phi}_0=0$, and in fact $\|\rg\phi_x\|\propto |x|$, since $\|\rg\phi_x\|^2 = \rg\kappa(x,x)=|x|^2\rho(0)$). 
This linear--radial RKHS is also defined over the input space $\mbb{U}$, denoted as $\mc{H}_u$ with a linear--radial product kernel $\rg{\vk}(u,u')=\vk_{\mr{lin}}(u,u')\vk_{\mr{sob}}(u,u')$ where $\vk_{\mr{sob}}(u,u')$ is a radial Sobolev kernel, and canonical features denoted as $\rg{\vf}_u = \rg{\vk}(u,\cdot) \in \rg{\mc{H}}_u$. 
Furthermore, in the next section, we will use the tensor product $\rg{\mc{H}}_x \otimes \rg{\mc{H}}_u$, which is an RKHS with kernel $((x, u), (x',u')) \mapsto \kappa(x,x') \varkappa(u,u')$. We denote the canonical feature of $(x,u)$ in this space as $\phi_x \otimes \varphi_u$.

\subsection{Kernel Quadratic Forms and Kernel SOS Forms}
\par The norm of a canonical feature satisfies $\lVert \rg{\phi}_x \rVert^2 = \rg{\kappa}(x,x)=\left( x^\top x \right) \rho(0)$, which scales with the $|x|^2$. This is clearly in contrast to the purely radial kernel, for which $\lvert \phi_x \rvert = \rho(0)$ for every $x$. Hence, the linear--radial RKHS is a natural space for constructing ``locally at least quadratic functions'', including Lyapunov candidates. 
We now define ``kernel quadratic forms'', used later for stabilization. In finite dimensions, a standard Lyapunov candidate is a quadratic form $v(x) = x^\top P x$, where $P$ is a postive semidefinte matrix. In the RKHS setting, the feature $\rg{\phi}_x$ plays the role of $x$ and a positive operator $P$ on the RKHS plays the commensurate role. For precision, the concepts of Hilbert--Schmidt operators are reviewed here.
\begin{definition}[Hilbert--Schmidt operators]
    A symmetric operator $Q=Q^*$ on a Hilbert space $\mc H$ is a \emph{Hilbert-Schmidt operator} if for an orthonormal basis $\left\{ u_i\right\}_{i\in\mN} \in \mc H$, we have
    \begin{equation*}
        \textstyle P=\sum_{i=1}^\infty \lambda_i u_i\times u_i
    \end{equation*}
    for some $\left\{ \lambda_i\right\}_{i\in\mN}\subset\mR$ with $\sum_i \lambda_i^2<\infty$. If, in addition, $\ip{h}{Ph} \geq 0$, $\forall h\in \mc H$, then $P$ is said to be positive and denoted as $P\in \mr{HS}_+(\mc H)$. 
\end{definition}
Now, given a positive Hilbert--Schmidt operator $P:\rg{\mc{H}}_x\rightarrow\rg{\mc{H}}_x$, we define the associated \emph{kernel quadratic form} as $x\mapsto \ip{\rg{\phi}_x}{P\rg{\phi}_x}$.
By the spectral decomposition in the above definition, it can be verified that $\ip{\rg{\phi}_x}{P\rg{\phi}_x} = \sum_{i=1}^\infty \alpha_iu_i(x)^2$, which is a (possibly infinite) \emph{sum of squared ``locally at least linear'' functions}.

\section{Bilinear Koopman Model}\label{sec:bilinear}
\par This section discusses the existence and data-driven estimation of a bilinear Koopman model for system \eqref{eq:system}. The advantages for such a model for control will become clear in the next section. Specifically, we are interested in when a Koopman operator $K: \rg{\mc{H}}_x \rightarrow \rg{\mc{H}}_x \oplus \rg{\mc{H}}_u \oplus \left( \rg{\mc{H}}_x \otimes \rg{\mc{H}}_u \right) : g \mapsto g \circ f$ is a well-defined, bounded linear operator. 
As such, its adjoint operator $M: \rg{\mc{H}}_x \oplus \rg{\mc{H}}_u \oplus \left( \rg{\mc{H}}_x \otimes \rg{\mc{H}}_u \right) \to \rg{\mc{H}}_x$ is \emph{a bilinear ``lifted model'' in infinite dimensions, mapping the state canonical feature, input canonical feature, and their tensor product to the canonical feature of the succeeding state}:
\begin{equation}\label{eq:bilinear}
    \rg\phi_{x^+} = \underbrace{\begin{bmatrix} M_1 & M_2 & M_3 \end{bmatrix}}_{M} \begin{bmatrix} \rg\phi_x \\ \rg\vf_u \\ \rg\phi_x\otimes \rg\vf_u \end{bmatrix} =: M\rg\psi_{(x,u)}.
\end{equation}
It is then shown how such a model can be estimated with data triplets $\left\{(u_i, x_i, y_i) \right\}_{i=1}^N$ where $y_i = f(x_i, u_i)$.

\subsection{Existence of A Bilinear Koopman Model}
Here we present the main result for the existence of an operator model in the form of \eqref{eq:bilinear}. 
\begin{theorem}\label{th:model}
    Assume that (i) $\mbb{X}\subset \mR^n$ and $\mbb{U}\subset \mR^m$ are compact and star-convex with respect to the origin, (ii) $f \in [\rg{C}^{s+1}(\mbb{X}\times\mbb{U})]^n$ with $s>(d_x+d_u)/2$, (iii) $f_x:=f(\cdot, 0)$ is non-degenerate: $\inf_{x\in \mbb{X}} |\operatorname{det} \mr{D}f_x(x)|>0$, and (iv) for every Borel set $S\subset \mbb{X}$ and $x\in \mbb X$, $\operatorname{mes} \{u\in \mbb U: f(x,u)\in S\}\leq c \operatorname{mes} S$ for some constant $c>0$. 
    Then there exists a bounded linear operator $M$ such that \eqref{eq:bilinear} holds for the system \eqref{eq:system}.  
\end{theorem}

\par The trick to prove the theorem is to examine the composition of $g \in \rg{\mc{H}}_x$ with ``restricted'' controlled dynamics. In particular, the restriction $x \mapsto f(x,0)$ captured purely state-dependent contributions, and $u \mapsto f(0,u)$ captures purely input-dependent contributions. The remaining component, which depends jointly on $x$ and $u$, is then associated with the tensor product feature $\rg{\phi}_x\otimes\rg{\vf}_u$. 
In the same spirit as \cite{tang-ye2025koopman}, the dynamics must be assumed to be smooth and equilibrium-preserving, for which we define
\begin{align*}
    \begin{aligned}
        \rg{C}^s(\mbb{X}\times\mbb{U})
        =
        \Bigg\{
        \sum_{i=1}^n e_i p_i
        +
        \sum_{i=1}^m e_{n+i} q_i
        \;:&\;
        p_i,q_i \in C^s(\mbb{X}\times\mbb{U}) \Bigg\},
    \end{aligned}
\end{align*}
which can be seen as the space of $s$-smooth functions that are locally linear around the equilibrium point. 

\par The composition of an observable $g \in \rg{\mc{H}}_x$ with the dynamics restricted to solely state contributions behaves the same as the case of an autonomous system, which has been analyzed in our previous work \cite{tang-ye2025koopman}.
\begin{lemma}
    Suppose that (i) $\mbb{X} \subset \mR^n$ is compact, (ii) $f \in \rg{C}^s(\mbb{X}\times\mbb{U})$, and (iii) $\inf_{x\in\mbb{X}}\left\lvert \det \mr{D}f_x(x) \right\rvert > 0$. Then on $\rg{\mc{H}_x}$, $\Lambda_1:g \mapsto g \circ f_x$ is a bounded linear operator. 
\end{lemma}

\par A similar result can also be proven in a similar manner for the composition $g \mapsto g \circ f_u$, where $f_u:=f(0,\cdot)$.
\begin{lemma}
    Suppose that (i) $\mbb{U}\subset \mR^m$ is compact, (ii) $f \in [\rg{C}(\mbb{X}\times\mbb{U})]^n$, and (iii) for any Borel set $S \subset \mbb{X}$, $\operatorname{mes}f_u^{-1}(S) \leq c \operatorname{mes} S$ for some constant $c>0$. Then $\Lambda_2 : \rg{\mc{H}}_x \rightarrow \rg{\mc{H}}_u : g \mapsto g \circ f_u$ is a bounded linear operator.
\end{lemma}
\begin{proof}
For any $g \in \rg{\mc{H}}_x$, we will verify that $g \circ f_u \in \rg{\mc{H}}_u$. It can be verified this composition can be written as:
$$ \textstyle (g \circ f_u) = \sum_{i=1}^mu_i\sum_{k=1}^ne_k\left( q_i(0,u) \right)g_k\left( f_u(u) \right). $$
Since by the second assumption, all $q_i \in C^s(\mbb{X}\times\mbb{U})$, this composition will be in $\rg{\mc{H}}_u$ if $g_k \circ f_u \in \rg{\mc{H}}_u$. By the definition of Sobolev spaces, 
\begin{equation*}
    \textstyle \norm{ g_k \circ f_u }_{H^s(\mbb{X})}^2 \leq c' \sup_{\lvert\alpha\rvert\leq s} \int_\mbb{U} \left\lvert \partial^\alpha g_k(f_u(u)) \right\rvert^2 \mr{d}u,
\end{equation*}
where $\alpha$ is a multi-index and $c'>0$. Using the chain rule, this partial derivative is expanded into a finite sum of terms, each containing a product of $\partial^\gamma g_k$ (for a multi-index $\lvert \gamma \rvert < \lvert \alpha \rvert$) and $\partial^\beta f_u$ (for some multi-indices $\lvert \beta \rvert < \lvert \alpha \rvert$). Again, with $f \in \rg{C}^s(\mbb{X}\times\mbb{U})$, this is bounded by $\sup_{\lvert \alpha \rvert\leq s} \int_\mbb{U} \left\lvert \partial^\alpha g_k(f_u(u)) \right\rvert^2 \mr{d}u$. Then, by condition (iii), a change of variables can be performed to yield $\int_{\mbb{X}}\lvert \partial^\alpha g_k(x)\rvert^2 \mr{d}\mu(x)$, where $\mu(S)=\operatorname{mes} f_u^{-1}(S)$ for any Borel set $S \subset \mbb{X}$. 
By the third assumption, this is linearly bounded by $\lVert g_k \rVert_{\rg{\mc{H}}_x}^2$.
\end{proof} 

\begin{lemma}
    Suppose that (i) $\mbb{X}\subset \mR^{n}$ and $\mbb{U}\subset \mR^{m}$ are compact and star-convex with respect to the origin, (ii) $f \in [\rg{C}^{s+1}(\mbb{X}\times\mbb{U})]^n$  with $s>(d_x+d_u)/2$, and (iii) for every Borel set $S\subset \mbb{X}$, $\operatorname{mes} f^{-1}(S)\leq c \operatorname{mes} S$. Then the operator $\Lambda_3: \rg{\mc{H}}_x^s \rightarrow \rg{\mc{H}}_x^s\otimes \rg{\mc{H}}_u^s : g \mapsto g \circ f - g \circ f_x - g \circ f_u$ is a bounded linear operator, where $f_x(x)=f(x,0)$ and $f_u(u)=f(0,u)$.
\end{lemma}
\begin{proof}
The composition $g \circ f$ can be written as
\begin{small}
\begin{equation*}
    g(f(x,u)) = \sum_{i=1}^{n} \left[ \sum_{k=1}^{n}x_i p_{ki}(x,u) + \sum_{j=1}^{m}u_j q_{ji}(x,u) \right] g_i(f(x,u)).
\end{equation*}
\end{small}

Defining functions $\bar p_k(x,u) := \sum_{i=1}^{n} e_i\left( p_{i}(x,u) \right) g_i(f(x,u))$ and $\bar q_j(x,u) := \sum_{i=1}^{n} e_i \left( q_{j}(x,u) \right) g_i(f(x,u))$, the composition $g \circ f - g \circ f_x - g \circ f_u$ can be expressed with
\begin{small}
\begin{equation*}
    (\Lambda_3g)(x,u)= \sum_{k=1}^{n} x_k \left[ \bar p_k(x,u)-\bar p_k(x,0) \right] + \sum_{j=1}^{m} u_j \left[ \bar q_j(x,u)-\bar q_j(0,u) \right]. 
\end{equation*}
\end{small}
By condition (i), $\bar p_k(x,u) - \bar p_k(x,0) = \sum_{j=1}^{m} u_j \int_0^1 \frac{\partial \bar p_k}{\partial u_j}(x,\tau u) \mr{d}\tau$ (with a similar result for $\bar q_j(x,u)$). Thus, this composition becomes $(\Lambda_3g)(x,u) = \sum_{k=1}^{n} \sum_{j=1}^{m} x_k u_j \bar s_{kj}(x,u)$, where $\bar s_{k,j}$ is the integral of first order derivatives of $\bar p_k$ and $\bar q_j$. It remains to show that for each $k$ and $j$, $\bar s_{kj} \in \rg{\mc{H}}_x \otimes \rg{\mc{H}}_u$. By the definitions of $\bar p_k$ and $\bar q_j$, their first derivatives are finite sums of products involving $p_{k i}$, $q_{j i}$, $f$, $g_i$ composed with $f$, and all of their first derivatives. We next estimate each coefficient $\bar s_{kj}$ in the tensor product Sobolev norm. For some $c'>0$,
\begin{equation*}
    \textstyle \lVert \bar s_{kj} \rVert_{H^s(\mbb{X})\otimes H^s(\mbb{U})}^2 \leq c' \sup_{\lvert \alpha\rvert\leq s} \int_{\mbb{X}\times\mbb{U}} \left\lvert \partial^\alpha \bar s_{kj}(x,u) \right\rvert^2 \mr{d}x\mr{d}u,
\end{equation*}
where $\alpha$ is a multi-index. By repeated application of the chain rule, each derivative $\partial^\alpha \bar s_{kj}$ is a finite sum of terms involving products of bounded derivatives of $f$, $p_{k r}$, and $q_{jr}$ with derivatives of the form $\partial^\gamma g_r(f(x,u))$, where $\lvert\gamma\rvert\leq s$. 
From condition (ii), $f\in\rg{C}^{s+1}(\mbb{X}\times\mbb{U})$ and thus $p_{k r}$ and $q_{jr}$ have the required bounded regularity, this gives
\begin{equation*}
    \textstyle \lVert \bar s_{kj} \rVert_{H^s(\mbb{X})\otimes H^s(\mbb{U})}^2 \lesssim c'  \sup_{\lvert\gamma\rvert\leq s} \int_{\mbb{X}\times\mbb{U}} \left\lvert \partial^\gamma g_r(f(x,u)) \right\rvert^2 \mr{d}x\mr{d}u
\end{equation*}
for some $c'>0$. By condition (iii) on the measure induced by $f$, with the same rationale as the previous lemma, 
\begin{equation*}
    \textstyle \lVert \bar s_{kj} \rVert^2_{H^s(\mbb{X})\otimes H^s(\mbb{U})} \leq c'\lVert g\rVert^2_{\rg{\mc{H}}_x}
\end{equation*}
for some $c'>0$. This result holds for all $k$ and $j$, and therefore $\Lambda_3$ is indeed a bounded linear operator. 
\end{proof}

\begin{proof}[Proof of Theorem \ref{th:model}]
    It suffices to prove that there exists an operator $M^*$ (which is the adjoint of the desired $M$) such that for any $g\in \rg{\mc{H}}_x$, $g(f(x, u)) = g(x^+) = \ip{g}{\phi(x^+)} = \ip{M^*g}{(\rg\phi_x, \rg\varphi_u, \rg\phi_x\otimes \rg\varphi_u)}$, for which we need three operators $M_1^*$, $M_2^*$, and $M_3^*$ such that $(M_1^*g)(x) + (M_2^*g)(u) + (M_3^*g)(x, u) = g(x^+)$. From the above lemmas we see that these are exactly $\Lambda_1$, $\Lambda_2$, and $\Lambda_3$. 
\end{proof}

\subsection{Learning from Snapshot Data}
\par Assume that the model \eqref{eq:bilinear} is well-defined. Given data triplets $\left\{ (x_i, u_i, y_i) \right\}_{i=1}^N$ where $y_i = f(x_i, u_i)$, the goal is to learn an operator $\hat M$ that acts as an approximate Perron--Frobenius operator, that is, $M\rg\psi_{(x_i,u_i)}$ matches $\rg\phi_{y_i}$ over the sample. This can be formulated as the following convex learning problem over Hilbert--Schmidt operators:
\begin{equation*}
    \min_{M \in \mr{HS}\left( \rg{\mc{H}}_{xu},\rg{\mc{H}}_x \right)} \frac{1}{N} \sum_{i=1}^N \left\lVert M\rg\psi_{(x_i,u_i)}-\rg\phi_{y_i} \right\rVert_{\rg{\mc{H}}_x}^2 + \beta \left\lVert M \right\rVert_{\mr{HS}}^2,
\end{equation*}
where $\beta>0$ is a regularization parameter and
$\rg{\mc{H}}_{xu} := \rg{\mc{H}}_x \oplus \rg{\mc{H}}_u \oplus (\rg{\mc{H}}_x \otimes \rg{\mc{H}}_u)$ is the Hilbert space associated with the feature $\rg{\psi}_{(x,u)}=(\rg{\phi}_x, \rg{\vf}_u, \rg{\phi}_x \otimes \rg{\vf}_u)$. 

\par By the representer theorem for learning on RKHS, and the fact that the space of Hilbert--Schmidt operators is Hilbert, the optimizer must have the form 
$$\textstyle M=\sum_{i=1}^N\sum_{j=1}^N \hat\theta_{ij}  \, \rg\phi_{y_i} \times \rg\psi_{(x_j,u_j)} , $$
where $\rg\phi_{y_i} \times \rg\psi_{(x_j,u_j)}$ is a rank-$1$ operator that maps any $\rg\psi_{(x,u)}$ to $\ip{\rg\psi_{(x_j,u_j)}}{\rg\psi_{(x,u)}}\rg\phi_{y_i}$. Thus, the optimization problem reduces to solving for a matrix $\hat\Theta = \left[ \hat\theta_{ij} \right]$. 
Let us denote $G_{\rg\psi} := [G_{\rg\psi, ij}]$ as the Gram matrix for the features $\psi_{(x,u)}$, in which $G_{\rg\psi, ij} = \ip{\rg\psi_{(x_i,u_i)}}{\rg\psi_{(x_j,u_j)}} = \rg\kappa (x_i,x_j) + \rg\vk(u_i, u_j) + \rg\kappa(x_i,x_j) \rg\vk(u_i, u_j)$, and $G_{\rg\phi} := [G_{\rg\phi, ij}]$, in which $G_{\rg\phi, ij} = \ip{\rg\phi_{y_i}}{\rg\phi_{y_j}} = \rg\kappa(y_i,y_j)$. It can be seen that the problem boils down to the follows:
\begin{equation}
    \min_{\Theta\in \mR^{N\times N}} \,
    \mr{tr}\bra{\frac{1}{N}G_{\rg\psi}^\top \Theta^\top G_{\rg\phi} \Theta G_{\rg\psi} - \frac{2}{N}  G_{\rg\psi}^\top \Theta^\top G_{\rg\phi} + \beta \Theta G_{\rg\psi}^\top \Theta^\top G_{\rg\phi}}
\end{equation}
which is indeed a convex optimization problem and actually allows an explicit solution through solving a generalized eigenvalue problem \cite{kostic2022learning, tang2025koopman}, which we omit for brevity here.

\subsection{Example}\label{subsec:example}
\par We consider the following example:
$$x_{t+1} = 0.95x_t - 0.2x_t^3 + u_t,$$
with state region $\mbb{X} = [-1, 1]$ and input space $\mbb{U} = [-1, 1]$. 
The operator $\hat{M}$ in the model \eqref{eq:bilinear} is learned from $N = 300$ i.i.d.\ uniform snapshots drawn from $\mathbb{X} \times \mathbb{U}$, and evaluated on $N_{\mathrm{test}} = 500$ test points.
The left subfigure of Fig.~\ref{fig1} is the training and testing data in this experiment. 
The right subfigure shows the predicted next state $\widehat{x^+}$ versus the true next state $x^+$ over the test set. The predictions $\widehat{x^+}$ align closely with the underlying true next state with a root mean squared error of $7.51 \times 10^{-3}$, verifying the accuracy of the learned operator. 
Here, the estimation is made by $\widehat{x^+} = \ip{e}{\hat{M} \rg\psi_{(x,u)}}$. 

\par As proven in \cite[Theorem 2.4]{bold2025kernel}, when the Koopman operator is estimated based on the radial Sobolev kernel, then for any function $g\in H^s(\mbb X)$, its estimated next-time value $\ip{g}{K^*\phi_x}$ with the Perron--Frobenius operator has an error bounded by a multiple of the fill distance to the power of $s-d/2$, and such an error is uniform over $x\in \mbb{X}$. 
Following the same idea, it is not hard to reason in our setting that, if using the linear--radial kernel instead of the radial kernel, such an error is proportional to $|(x,u)|$ multiplied by the fill distance to the power of $s-(d_x+d_u)/2$. 
Because the fill distance scales at the rate of $N^{-1/(d_x+d_u)}$, we have
$$|\hat{x}-x|\leq C |(x,u)|, \enspace C=O\bra{ N^{-\bra{\frac{s}{d_x+d_u} - \frac{1}{2}}} }. $$
The left subfigure of Fig.~\ref{fig3} visualizes the prediction error
$x_i^+ - \hat{x}_i^+$ at each test point $(x_i, u_i)$ as scatter
points, alongside the error bound surfaces $\pm C\left|(x_i, u_i)\right|$.
The bound constant $C$ is estimated empirically as the maximum $|x_i-\hat{x}_i|/|(x_i, u_i)|$ over the sample, which we obtain $0.1368$ as its value. 
Having used $s=(d_x+d_u)/2+k+1/2$ ($k=2$) \cite{wendland2004scattered} in this experiment, we have theoretically $C=O(N^{-5/4})$. The right subfigure of Fig.~\ref{fig3} shows the variation of $C$ as a function of training size $N \in \{50, 100, 150, 200, 250, 300\}$, in comparison to the gray dashed reference line $N^{-5/4}$.

\begin{figure}
    \centering
    \includegraphics[width=\linewidth]{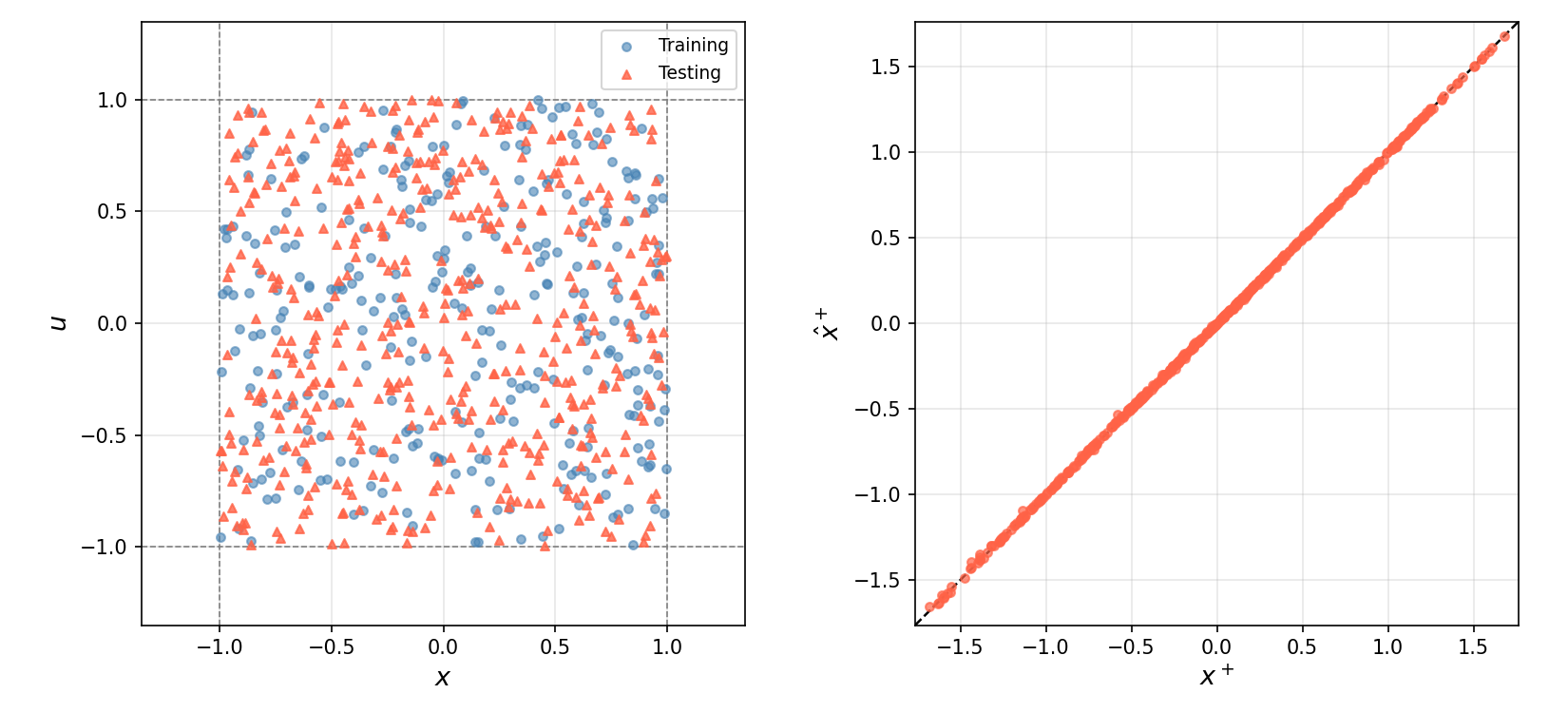}
    \caption{(Left) Training and testing data points $(x, u)$; (Right) The predicted state $\hat{x}^+$ and underlying $x^+$ using the learned operator.} 
    \label{fig1}
\end{figure}
\begin{figure}
    \centering
    \includegraphics[width=\linewidth]{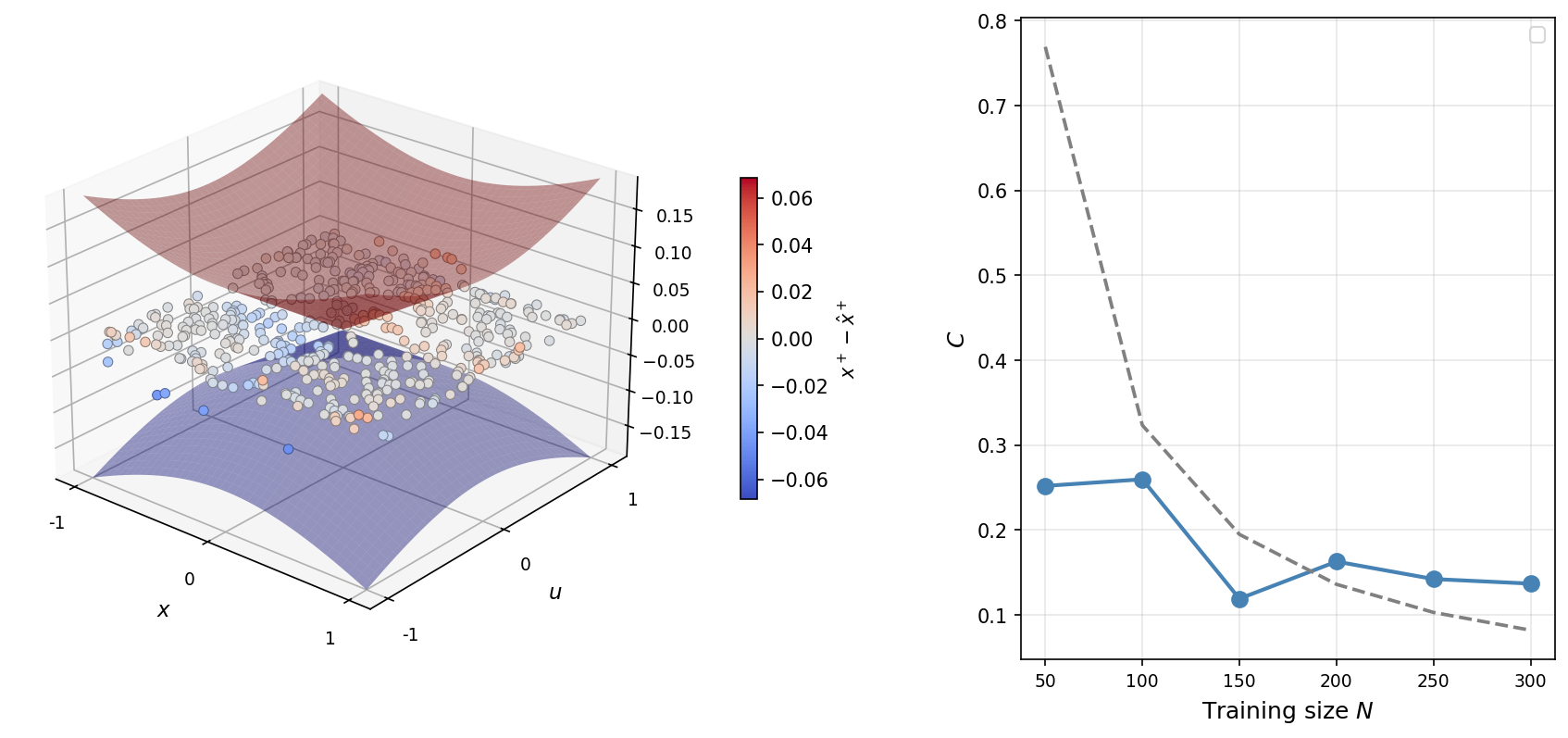}
    \caption{(Left) The prediction error $\hat{x}^+ - x^+$ and error bound at $N=300$; (Right) The empirical $C$ in terms of the training data size $N$.} 
    \label{fig3}
\end{figure}

\section{Koopman-Based Stabilizing Control under Bilinearity and Input Lifting}\label{sec:stabilization}
The advantage of having a model like \eqref{eq:bilinear} where the inputs and states are separately lifted into two RKHSs is to facilitate the formulation of control problems. Here, we consider the problem of synthesizing a Lyapunov-based controller. 
Specifically, we suppose that a Lyapunov function $v: \mbb{X}\to \mR_+$, $v(x) = \ip{\rg\phi_x}{P\rg\phi_x}$ is assigned, where $P \in \mr{HS}_+(\rg{\mc{H}}_x)$ and look for a feedback control law $u = \alpha(x)$ with certain regularity conditions, such that the Lyapunov function decreases along the closed-loop dynamics faster than a prescribed decay rate function $q(x) = \ip{\rg\phi_x}{Q\rg\phi_x}$, where $Q \in \mr{HS}_+(\rg{\mc{H}}_x)$.\footnote{
The problem can, of course, be posed in alternative ways, such as requiring an  linear rate of decay, or maximizing the coefficient $c\geq 0$ under a required decay rate of $cq(x)$. }

\subsection{Lyapunov Descent under Stochastic Feedback Laws}
The problem, hence, is to find a law $\alpha$ such that 
$$\ip{M\rg\psi_{(x, \alpha(x))}}{PM\rg\psi_{(x, \alpha(x))}} \leq \ip{\rg\phi_x}{(P-Q)\rg\phi_x}, \, \forall x\in \mbb{X}.$$
Such a problem, however, suffers from the issue that the left-hand-side term and the right-hand-side term cannot be agglomerated due to the nonlinear input feature map $\rg\vf$ contained in $\rg\psi$. 
In other words, even if for every $\rg\phi_x$ we can find an element in $\rg{\mc H}_u$ that satisfies the Lyapunov inequality, we do not have a criterion to verify that this element has a corresponding finite-dimensional admissible control action $u\in \mbb{U}$. 
Essentially, the difficulty here is the lack of convexity in nonlinearly constrained optimization problems. 

\par Motivated by the wide use of nonconvex optimization in machine learning problems and the need to interpret the trained models' behavior, a ``continuum'' perspective, viewing the nonconvex problem as the optimization over probability distributions, has been increasingly adopted in the recent literature \cite{marteau2020non, rudi2025finding}. 
Specifically in our problem, we consider the search of a stochastic feedback law, where given any $x$, $u$ is a random vector, instead of a deterministic mapping, where $u$ is fixed under $x$. With such a consideration, by writing the conditional distribution of input as the probability measure $\mu(\cdot|x)$, we write the Lyapunov function value at the succeeding time averaged over all succeeding states as 
$$\textstyle \mbb{E}[v(x^+)] = \int_{\mbb{U}} \ip{M\rg\psi_{(x,u)}}{PM\rg\psi_{(x,u)}} \mu(\xD{u}|x). $$
The search is then performed on such conditional distributions $\mu(\cdot|\cdot)$ defined for $x\in \mbb{X}$ and $u\in \mbb{U}$. 

\par The resulting problem is clearly convex when regarding the measure $\mu$ as the decision, since it is searched in a linear space subject to linear constraints. 
In fact, none is lost by randomizing the feedback law, since if for any $x$, a deterministic feedback law $u=\alpha(x)$ achieves the Lyapunov descent, then the Dirac conditional distribution $\mu(\cdot|x)=\delta(\cdot-\alpha(x)|x)$ over $\mbb{U}$ satisfies the need. 
Vice versa, if a condition distribution $\mu(\cdot)$ achieves the descent, then there must be an appropriate value for $u$ under every $x$, which forms a deterministic feedback law. 
\begin{problem}\label{prob:stabilization.distribution}
    Find a conditional distribution $\mu(\xD{u}|x)$ such that for all $x\in \mbb{X}$, we have 
    \begin{equation}\label{eq:descent.distribution}
        \textstyle \int_{\mbb{U}} \ip{M\rg\psi_{(x,u)}}{PM\rg\psi_{(x,u)}} \mu(\xD{u}|x) \leq \ip{\rg\phi_x}{(P-Q)\rg\phi_x} .  
    \end{equation}
\end{problem}

\subsection{Optimization over Conditional Distributions}
For a tractable approximation of \eqref{eq:descent.distribution}, we must perform a discretization of the conditional distribution $\mu(\xD{u}|x)$. That is, dependent on $x$, we should choose a sufficiently large amount of points $\{u_j(x)\}_{j=1}^{N_j(x)} \subset \mbb{U}_x$ and let $\mu(\xD{u}|x) = \sum_{j=1}^{N_j(x)} w_j(x) \delta_{u_j(x)}$, with all $w_j(x)$ being nonnegative and summing up to $1$, thus converting the condition \eqref{eq:descent.distribution} to 
$$ \textstyle \sum_{j=1}^{N_j(x)} \ip{M\rg\psi_{(x,u_j(x))}}{PM\rg\psi_{(x,u_j(x))}} w_j(x) \leq \ip{\rg\phi_x}{(P-Q)\rg\phi_x}. $$
Denoting $f(x, u_j(x))=:x_j^+$, the above inequality becomes
\begin{equation}\label{eq:descent.x.pointwise}
    \textstyle  \sum_{j=1}^{N_j(x)} \ip{\rg\phi_{x_j^+}}{P \rg\phi_{x_j^+}} w_j(x) - \ip{\rg\phi_x}{(P-Q)\rg\phi_x} \leq0.    
\end{equation}

To analyze the soundness of such a discretization on $\mbb{U}$, we should bound the possible loss when clipping a general distribution (given $x$) onto a discrete distribution. 
Such a loss clearly depends on the Lipschitz continuity of the Lyapunov function $v(f(x, u))$ on $u \in \mbb{U}_x$ (denoted as $\mr{lip}(v(f(x, \cdot)), \mbb{U}_x)$) and the fill distance of $\{u_j(x)\}$ on $\mbb{U}_x$ (denoted as $\mr{fill}(\{u_j(x)\}, \mbb{U}_x)$). 
Specifically, if a distribution $\mu(\xD{u}|x)$ supported on $\mbb{U}_x$ exists to satisfy \eqref{eq:descent.distribution}, then \eqref{eq:descent.x.pointwise} holds with its right-hand-side $0$ modified by $\mr{lip}(v(f(x, \cdot)), \mbb{U}_x) \times \mr{fill}(\{u_j(x)\}, \mbb{U}_x)$. On the contrary, if \eqref{eq:descent.x.pointwise} holds, then the discrete distribution directly meets \eqref{eq:descent.distribution}. 

\par It is necessary to preassign the shape of $\mbb{U}_x$ under any $x$. Note that $\ip{M\rg\psi_{(x, u)}}{PM\rg\psi_{(x, u)}}$ is a positive semidefinite quadratic form over $\{\rg\vf_u : u\in \mbb{U}_x\}$. If the form is further strictly positive semidefinite, then in order that this form does not exceed $\ip{\rg\phi_x}{(P-Q)\rg\phi_x}$, $\rg\vf_u$ must be norm-bounded by a constant multiple of $\|\rg\phi_x\|$. 
Recalling that $\|\rg\vf_u\| \propto |u|$ and $\|\rg\phi_x\|\propto |x|$, any input action causing Lyapunov descent at $x$ must be bounded in the ball of radius $c|x|$ centered at $0$ in $\mbb{U}$ (for some constant $c>0$). Hence, let $\mbb{U}_x = \mbb{U} \cap \mbb{B}(c|x|) \subset \mbb{U}$ satisfies the need. 
As such, the sampling strategy on $\mbb{U}_x$ can simply be scaling a fixed set of sample points on the unit ball by a constant of $c|x|$, i.e., $u_j(x) = (c|x|)\ms{u}_j$, $\{\ms{u}_j\}_{j=1}^{N_j} \subset \mbb{B}(1)$. The fill distance of $\{u_j(x)\}$in $\mbb{U}_x$ is $c|x|$ multiplied by the fill distance of $\{\ms{u}_j\}$ in $\mbb{B}(1)$, which is a constant.  

\par We further handle the $x$-pointwise constraints \eqref{eq:descent.x.pointwise} by a discretization on $\mbb{X}$. Choosing a number of sample points $\{x_i\}_{i=1}^{N_i} \subset \mbb{X}$, we postulate the following form for the weight functions $w_j(x)$:
$$\textstyle w_j(x)=\sum_{i=1}^{N_i} \theta_{ij} \varrho(x,x_i), \enspace j = 1,\cdots, N_j, $$
where $\varrho$ is a Gaussian kernel, i.e., $\varrho(x,x')=\exp(-|x-x'|^2/2\sigma^2)$ for some constant $\sigma>0$. We require $\theta_{ij}\geq 0$ for nonnegativity. We hence impose the constraint \eqref{eq:descent.x.pointwise} on the sampled points $\{(x_i, u_j(x_i), x_{ij}^+=f(x_i, u_j))\}$:
\begin{displaymath}
   \textstyle  \sum_{i'=1}^{N_i}\sum_{j=1}^{N_j} \theta_{i'j} \varrho(x_i,x_{i'}) \ip{\rg\phi_{x_{ij}^+}}{P \rg\phi_{x_{ij}^+}} - \ip{\rg\phi_{x_i}}{(P-Q)\rg\phi_{x_i}} \leq0.   
\end{displaymath}
That is, 
\begin{equation}\label{descent.x.sampled}
   \textstyle  \sum_{i'=1}^{N_i}\sum_{j=1}^{N_j} \theta_{i'j} v(x_{ij}^+) - (v(x_i)-q(x_i))\leq 0. 
\end{equation}
In addition to $\theta_{ij}\geq 0$, for normalization, we require that: 
$$\textstyle  \sum_{i'=1}^{N_i}\sum_{j=1}^{N_j} \theta_{i'j} \varrho(x_i,x_{i'}) = 1, \enspace i=1,\cdots,N_i.$$

\par Finally, by solving all the coefficients $\{\theta_{ij}\}$, we take the deterministic feedback policy as 
$$\textstyle \alpha(x) = \sum_{j=1}^{N_j} \sum_{i=1}^{N_i} \theta_{ij}\varrho(x, x_i) u_j(x). $$

\subsection{Example}
For the example problem in \S\ref{subsec:example}, setting $v(x)= x^2$ as the Lyapunov function, we consider the controller law under which the Lyapunov function decays by $q(x) = 0.19 x^2$, so that the required descent condition is
    \begin{equation*}
        v(f(x, \alpha(x))) \leq  0.81\,x^2, \quad \forall x \in \mathbb{X}.
    \end{equation*}
We sample
$N_i = 200$ state points $\{x_i\}_{i=1}^{N_i}$ uniformly and independently
from $\mathbb{X} = [-1, 1]$.
For each state $x_i$, a set of $N_j = 50$ input candidates is constructed as $u_j(x_i) = c|x_i|\,\mathsf{u}_j, \quad j = 1, \ldots, N_j$,
where $\{\mathsf{u}_j\}_{j=1}^{N_j}$ is a fixed uniform grid on $[-1, 1]$
and $c = 0.5$.

By solving linear programming problem according to~\eqref{descent.x.sampled} and setting $\sigma=0.2$ for the Gaussian kernel $\varrho(x,x')$, the feedback policy is constructed and shown in the left subfigure of Fig.~\ref{fig2}. The policy is odd-symmetric and vanishes at the origin, consistent with the equilibrium-preserving structure of the kernel. The right subfigure of Fig.~\ref{fig2} shows the closed-loop next state $x^+ = f(x, \alpha(x))$ versus $x$, with the shaded region indicating the descent band $|x^+| \leq 0.9|x|$. The closed-loop trajectory lies entirely within this band, 
confirming that the 
Lyapunov descent condition is satisfied across the entire $\mathbb{X}$. 
\begin{figure}
    \centering
    \includegraphics[width=1\linewidth]{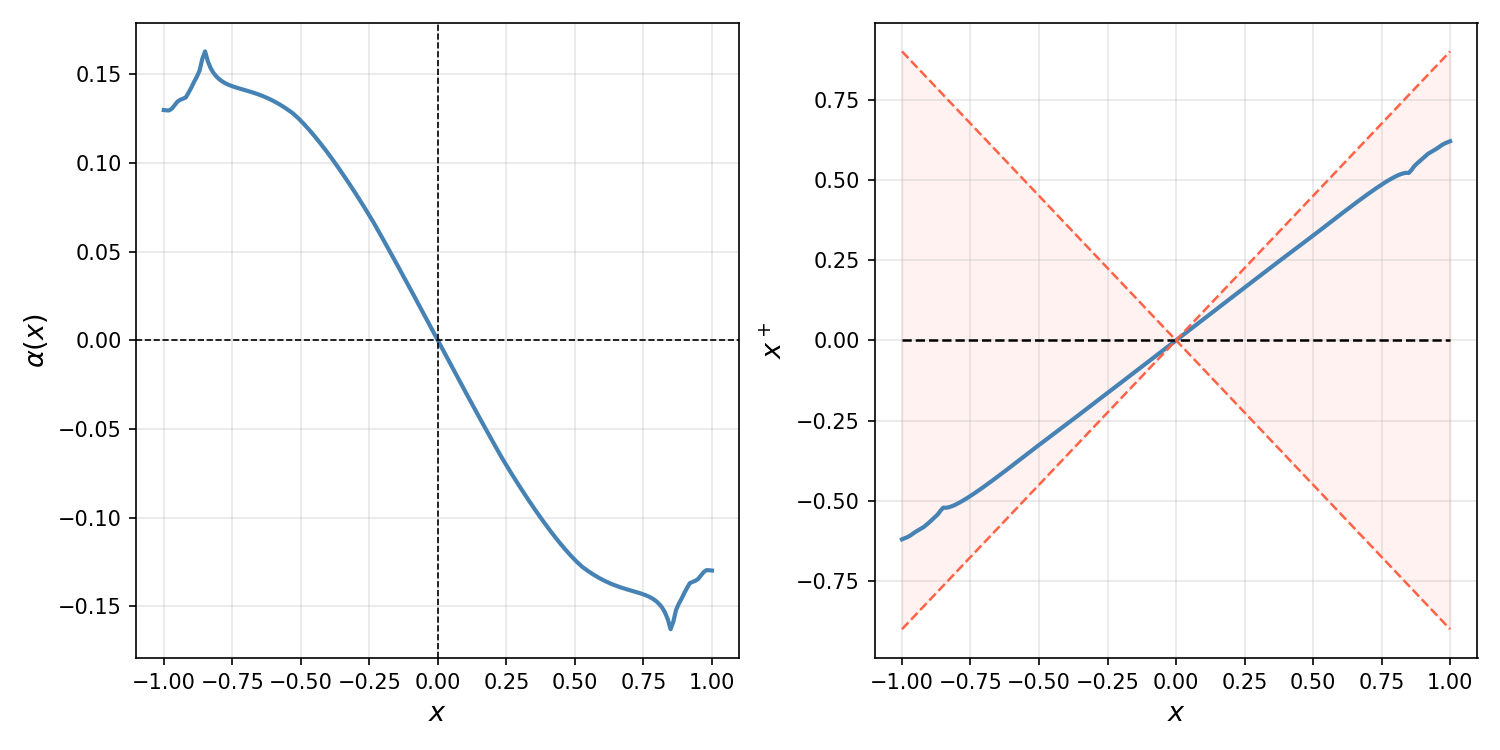}
    \caption{(Left) Control policy $\alpha(x)$; (Right) $x$ and $x^+$ under the policy.}
    \label{fig2}
\end{figure}

\section{Conclusions and Open Problems}\label{sec:conclusion}
In conclusion, the present paper has demonstrated that nonlinear discrete-time systems, under certain smoothness conditions, can be exactly represented as a bilinear system where the state and input are separately lifted into an RKHS. As the RKHSs are generated by linear--radial kernels, kernel SOS forms are naturally expressive for Lyapunov functions, thus providing the condition for Lyapunov function descent. 
The determination of a stabilizing controller, for a convex synthesis formulation, is cast as the optimization over a state-dependent probability density function over the input space. The solution approach is data-driven. 

\par Due to the nonlinear input feature map $\rg\varphi$, it appears unrealistic to tractably optimize the control law in its original deterministic formulation in finite dimensions. 
This is unavoidable in discrete-time systems, where it cannot be anticipated that the input feature map can be a linear one. 
The optimal control problem, which we have not addressed in this paper, if formulated under a kernel SOS form of stage cost, must also be solved via the ``stochastic mixture'' path. 
Structurally, this will appear similar to a continuous-time stochastic optimal control problem, whose operator-theoretic solution has been considered for stochastic differential equations with a Wiener noise \cite{houska2025convex}. 
For system \eqref{eq:bilinear}, the evolution of the density of states and inputs likely should be characterized by their averaged features (i.e., kernel mean embeddings). 

\par For continuous-time systems, one can be more optimistic about the possibility of a fully convex operator-theoretic solution to determine (approximately in a data-driven manner) the optimal controller. Formally, the Koopman model can be made a bilinear one without using input lifting. 
Echoing with our Footnote \ref{fn:1}, theoretical efforts are still needed to rigorously formulate the stabilizing control and optimal control problems. 

\bibliographystyle{ieeetr}
\bibliography{mybib}{}

\end{document}